\documentclass[11pt]{amsart}
\usepackage{amssymb}

\newcommand{\N}{\mathbb{N}}
\newcommand{\R}{\mathbb{R}}
\newcommand{\Rb}{\overline{\mathbb{R}}}
\newcommand {\F} {\mathcal{F}} 
\newcommand{\e}{{\rm e-\!}}
\newcommand{\M}{{\rm M-\!}}

\newcommand{\ps}{\smallbreak}

\newcommand{\tos}{\rightrightarrows}

\newtheorem{theo}{Theorem}[section]

\newtheorem{prop}[theo]{Proposition}
\newtheorem{fact}[theo]{Fact}

\newtheorem{cor}[theo]{Corollary}

\begin{document}
\thispagestyle{empty}

\title[Limits of maximal monotone operators]{Limits of maximal monotone operators driven by their representative functions}
\date{}
\author{Yboon Garc\'{\i}a}
\address{CIUP - Centro de Investigaci\'on de la Universidad del Pacı\'\i fico, Lima,
Per\'u}
\email{garcia\_yv@up.edu.pe}
\author{Marc Lassonde}
\address{Universit\'e des Antilles, Pointe \`a Pitre, and LIMOS, Clermont-Ferrand,
France}
\email{marc.lassonde@univ-antilles.fr}

\subjclass[2010]{47H05, 49J52, 47N10}
\keywords{Maximal monotone operator, convex function, representative function,
epi-convergence, Mosco-convergence, subdifferential}

\begin{abstract}
In a previous paper, the authors showed that in a reflexive Banach space the 
lower limit of a sequence of maximal monotone operators
is always representable by a convex function. The present paper gives
precisions to the latter result by demonstrating
the continuity of the representation
with respect to the epi-convergence of the representative functions,
and
the stability of the class of maximal monotone operators with respect
to the Mosco-convergence of their representative functions.
\end{abstract}

\maketitle
\section{Introduction}
Let $X$ be a real Banach space.
We denote by $X^*$ its topological dual and by 
$\langle .,.\rangle$ the duality product in $X\times X^*$,
that is,
$$\langle x,x^*\rangle:=x^*(x).$$
The product space $X\times X^*$ is assumed to be equipped with the
norm topology. Whenever necessary, the space $X$ is considered as a subspace
of $X^{**}$, so that $X^*\times X$ is a subspace the topological dual of $X\times X^*$.
The \textit{indicator function} of $A \subset X$ is the function
$\delta_A : X \to \R \cup \{+ \infty\}$ given by
$$
\delta_A(x) = \left\{ \begin{array}{cl}
         0    &  \mbox{if}\quad x\in A\\
        +\infty &  \mbox{if}\quad x\not\in A,
       \end{array}
                                 \right.
$$

We use the notation $\Gamma (X)$ for the space of all
lower semicontinuous (lsc) convex functions from $X$ into the extended real line
$\R \cup \{+ \infty\}$.
The \textit{conjugate} of $f\in \Gamma (X)$ is the function
$f^*:X^*\to \Rb$ given by
$$
f(x^*):=\sup_{x\in X} (\langle x,x^*\rangle-f(x)).
$$

A set-valued operator $T:X\tos X^*$, or a subset $T\subset X\times X^*$,
is called
\ps
$\bullet$ {\it monotone} provided
$
\langle y-x,y^*-x^*\rangle \geq 0,\  \forall (x,x^*),(y,y^*)\in T,
$
\ps
$\bullet$ {\em maximal monotone} provided it is monotone and
maximal (under set inclusion) in the family of all monotone sets contained in $X\times X^*$,
\ps
$\bullet$ {\em representable} \cite{M-LS05}
provided there is a function $f\in\Gamma(X\times X^*)$
such that
$$
\left\lbrace\begin{array}{l}
f(x,x^*)\ge \langle x, x^* \rangle,\quad \forall \,(x,x^*) \in X \times X^*,\label{equ11}\smallskip\\
f(x,x^*)= \langle x, x^* \rangle \Leftrightarrow (x,x^*) \in T.\label{equ12}
\end{array}\right.
$$
We denote by
$$
\mathcal{F}(X\times X^*):=
\{ f\in \Gamma(X\times X^*) : f(x,x^*)\ge \langle x, x^* \rangle\ \forall \,(x,x^*)
\in X\times X^*\},
$$
the class of all \textit{representative} functions
and for $f\in\mathcal{F}(X\times X^*)$, we denote by
$$L(f):=\{(x,x^*) \in X\times X^* : f(x,x^*)= \langle x, x^* \rangle\}$$
the graph of the set-valued map \textit{represented} by $f$.
The class of representable operators can then be
synthetically expressed as:
\begin{equation*}
T \subset X\times X^*
\mbox{ is representable }\Longleftrightarrow 
\exists f\in \mathcal{F}(X\times X^*) \,:\, T=L(f).
\end{equation*}

Every representable operator is easily seen to be monotone;
the converse is not true.
For a nonempty $T:X\tos X^*$, we define $\varphi_T\in\Gamma(X\times X^*)$
by
$$
\varphi_T(x,x^*)
:= \sup_{(y,y^*)\in T}\langle y-x, x^*-y^*\rangle+\langle x, x^*\rangle.
$$
This function was introduced by Fitzpatrick \cite{Fit88} who showed that
it represents $T$ whenever $T$ is maximal monotone, in which case
it is actually the smallest representative function of $T$.
See, e.g., \cite{Fit88,M-LS05,GL12} and the references therein
for more details.
Moreover, we have:
\begin{fact}\cite[Theorem 3.1]{BS03}  \label{BS03}
Let $X$ be a reflexive Banach space and let $f\in \mathcal{F}(X\times X^*)$.
Then, $L(f)$ is maximal monotone
if and only if
$f^*(x^*,x)\ge \langle x, x^* \rangle$ for all $(x,x^*)\in X\times X^*$ .
\end{fact}
In view of Fact \ref{BS03}, it is natural to also consider the class
\begin{multline*}
\mathcal{F^*}(X\times X^*):=
\{ f\in \mathcal{F}(X\times X^*) : 
\ f^*(x^*,x)\ge \langle x, x^* \rangle\ \forall \,(x,x^*)\in X\times X^*\}
\end{multline*}
of all representative functions of maximal monotone operators.
Fact \ref{BS03} can then be rewritten as:
in a reflexive Banach space $X$,
\begin{equation*}
T \subset X\times X^*
\mbox{ is maximal monotone }\Longleftrightarrow 
\exists f\in \mathcal{F^*}(X\times X^*) \,:\, T=L(f).
\end{equation*}

\smallbreak
Let $(T_n) \subset X \times X^*$ be a sequence of operators
from a Banach space $X$ to its dual $X^*$.
The \textit{lower limit} of $(T_n)$,
with respect to the strong topologies in $X$ and $X^*$, is the operator
$$
\liminf T_n:=
\{\lim_n \, (x_n,x^*_n): (x_n,x^*_n) \in T_n,\, \mbox{for all } n \in \N \}.
$$

A sequence $(f_n)\subset \Gamma(X)$ is said to {\em epi-converge\/}
to $f$, written $f=\e\lim_n f_n$, provided that at each point $x\in X$
both of the following conditions are satisfied:

(i) there exists $(x_n)$ convergent strongly to $x$ with
$f(x)=\displaystyle\lim_n f_n(x_n)$;

(ii) whenever $(x_n)$ is strongly convergent to $x$,
we have $ \displaystyle f(x)\le \liminf_n f_n(x_n)$.

A sequence $(f_n)\subset \Gamma(X)$
is said to {\em Mosco-converge\/} 
to $f$, written $f=\M\lim_n f_n$, provided that at each point $x\in X$
both of the following conditions are satisfied:

(i) there exists $(x_n)$ convergent strongly to $x$ with
$f(x)=\displaystyle\lim_n f_n(x_n)$;

(ii) whenever $(x_n)$ is weakly convergent to $x$,
we have $ \displaystyle f(x)\le \liminf_n f_n(x_n)$.

Of course, in a finite dimensional space $X$, epi and Mosco convergences are the same.

\begin{fact}\cite[Theorem 1]{Mos71}\label{Mos71}
Let $X$ be a reflexive Banach space, $(f_n)\subset \Gamma(X)$ and $f\in \Gamma(X)$.
Then, $f=\M\lim_n f_n$ if and only if $f^*=\M\lim_n f^*_n$.
\end{fact}

Let $X$ be a reflexive Banach space supplied
with a strictly convex norm whose dual norm on $X^*$ is strictly convex.
Under these assumptions, the {\it duality mapping} $J:X\tos X^*$ given by
$$
Jx:= \bigl\{x^*\in X^*: \langle x^*, x \rangle = \|x\|^2 = \|x^*\|^2\bigr\}
$$
is single-valued, bijective and maximal monotone, and 
a monotone operator $T:X\tos X^*$ is maximal monotone if and only if
the operator $J+T$ is onto (see, e.g., Rockafellar \cite{Roc70}).
Moreover, in that case, the operator
$(J+T)^{-1}$ is single-valued on $X^*$.
According to this result, for any maximal monotone operator $T_n: X \tos X^*$
and any $(x,x^*) \in X \times X^*$,
there is a unique solution $x_n=J_{T_n}(x,x^*)$ of the inclusion
\begin{equation} \label{reussi1}
 x^* \in J(x_n-x) + T_n(x_n).
\end{equation}
\begin{fact}\cite[Lemma 3.1 and Proposition 3.2]{GL12}\label{caracliminf}
Under the above assumptions on $X$,
let $(T_n)\subset X \times X^*$ be a sequence of maximal monotone operators.
\smallbreak
{\rm (1)} If $\,\liminf T_n$ is not empty, then for every $(x,x^*) \in X \times X^*$
the sequence $(J_{T_n}(x,x^*))_n$ of solutions of the inclusion (\ref{reussi1})
is bounded.
\smallbreak
{\rm (2)} $\liminf T_n=\{(x,x^*)\in X \times X^* :  J_{T_n}(x,x^*)\to x \}$.
\end{fact}

Using the above Fact \ref{caracliminf},
we showed in \cite[Theorem 3.3]{GL12} that, in a reflexive Banach space, the 
lower limit $T=\liminf T_n$ of a sequence of maximal monotone operators $(T_n)$
is always representable. The aim of the present paper is to bring
precisions to this result. In Theorem \ref{reussi12017} below,
we establish:

(1) the continuity of the operation $f \mapsto L (f)$
with respect to the epi-convergence of the representative functions
in $\mathcal{F^*}(X\times X^*)$,

(2) the stability of the class of maximal monotone operators with respect
to the Mosco-convergence of their representative functions.

As an application of this theorem, in the last section we establish analogous results
concerning the monotone operators defined as diagonals of subdifferentials of convex
functions. 

\section{Convergence of representative functions}

We first observe that the class $\mathcal{F}(X\times X^*)$ is stable
under epi-convergence and the class $\mathcal{F^*}(X\times X^*)$ is stable
under Mosco-convergence:

\begin{prop}\label{convergence}
Let $X$ be a reflexive Banach space.
\smallbreak
{\rm (1)} Let $(f_n)\subset \mathcal{F}(X\times X^*)$.
If $(f_n)$ epi-converges to $f$,
then $f\in \mathcal{F}(X\times X^*)$.
\smallbreak
{\rm (2)} Let $(f_n)\subset \mathcal{F^*}(X\times X^*)$.
If $(f_n)$ Mosco-converges to $f$,
then $f \in \mathcal{F^*}(X\times X^*)$.
\end{prop}

\begin{proof}
(1)
It follows directly from the definition of epi-convergence that
$f$ belongs to $\Gamma (X\times X^*)$.
Let $(x,x^*) \in X\times X^*$. By (i) in the definition of epi-convergence,
there exists a sequence $(x_n,x_n^*)$ converging to $(x,x^*)$ such that
$\lim_nf_n(x_n,x_n^*) = f(x,x^*)$.
Since $f_n\in \F (X\times X^*)$ for every $n\in \N$, one has
$\langle x_n, x_n^*\rangle \leq f_n(x_n,x_n^*)$
for every $n\in \N$.
Hence, $\langle x, x^*\rangle \leq f(x,x^*)$.
This shows that $f\in\F(X\times X^*)$.
\smallbreak
(2)
By Fact \ref{Mos71}, we have both
$f=\e\lim f_n$ with $(f_n)\subset \mathcal{F}(X\times X^*)$
and $f^*=\e\lim f^*_n$ with
$(f^*_n)\subset \mathcal{F}(X^*\times X)$.
It follows from (1) that $f(x,x^*)\ge \langle x, x^*\rangle$
and $f^*(x^*,x)\ge \langle x, x^*\rangle$
for every $(x,x^*)\in X\times X^*$.
Hence,
$f\in \mathcal{F^*}(X\times X^*)$.
\end{proof}

\begin{theo} \label{reussi12017}
Let $X$ be a reflexive Banach space and let $(f_n) \subset \F (X\times X^*)$
with $T_n:=L(f_n)$ maximal monotone.
\smallbreak
{\rm (1)} If $(f_n)$ epi-converges to $f$, then $L(f)=\liminf T_n$.
\smallbreak
{\rm (2)} If $(f_n)$ Mosco-converges to $f$, then $L(f)=\liminf T_n$
and $\liminf T_n$ is maximal monotone.

\end{theo}
\begin{proof}
According to Asplund \cite{Asp67}, there exists an equivalent strictly convex norm
on $X$ whose dual norm on $X^*$ is strictly convex.
Since the statements to be proved do not depend on the norms in $X$ and $X^*$,
without loss of generality we can assume that the given norm on $X$
is strictly convex as well as its dual. 
This allows us to use Fact \ref{caracliminf}.
\smallbreak
(1)
By Proposition \ref{convergence}, $f\in\F(X\times X^*)$.
Let $T:=\liminf T_n$.
We claim that $T\subset L(f)$.
Indeed, let  $(x,x^*) \in T$. By definition of $\liminf_n T_n$,
for each $n\in\N$ there exists $(x_n,x_n^*) \in T_n$
such that $(x_n,x_n^*)\to (x,x^*)$.
Since $T_n=L(f_n)$, for every $n\in \N$ one has
$f_n(x_n,x_n^*) = \langle x_n, x_n^*\rangle$.
But it follows from the definition of $\e\lim_n f_n$ that
$\liminf_nf_n(x_n,x_n^*) \ge f(x,x^*)$, hence
\begin{equation*}
\langle x, x^*\rangle=
\liminf_n\langle x_n, x_n^*\rangle=\liminf_nf_n(x_n,x_n^*) \ge f(x,x^*).
\end{equation*}
Since $f\in\F(X\times X^*)$ we also have $\langle x, x^*\rangle \leq f(x,x^*)$.
So $f(x,x^*)= (x, x^*)$, proving that $(x,x^*) \in L(f)$.

\smallbreak
It remains to show that $L(f)\subset T$.
Let $(x,x^*) \in L(f)$. For each $n\in\N$ let $x_n=J_{T_n}(x,x^*)$ be the solution
of the inclusion (\ref{reussi1}).
According to Fact \ref{caracliminf}\,(2), to prove that $(x,x^*) \in T$
it suffices to show that $x_n \to x$.
Let $((y_n,y_n^*))_n$ be a sequence in $X\times X^*$ such that
$\lim_nf_n(y_n,y_n^*) = f(x,x^*)$.
Since $f_n$ is a representative function of $T_n$, for every $n\in \N$ one has
$$
f_n(y_n,y_n^*) \geq \varphi_{T_n}(y_n,y_n^*).
$$
Hence
\begin{align}\label{one}
\langle x, x^*\rangle=f(x,x^*)=\lim f_n(y_n,y_n^*)
\geq \limsup \varphi_{T_n}(y_n,y_n^*).
\end{align}
Since $x^* -J(x_n-x)$ belongs to $T_n(x_n)$,
 we derive from the definition of $\varphi_{T_n}$ that for every $n\in \N$,
\begin{align*}
\varphi_{T_n}(y_n,y_n^*)  &\geq \langle y_n-x_n ,x^* -J(x_n-x) - y_n^*\rangle + \langle y_n, y_n^*\rangle\\
&=\langle y_n-x_n,x^* - y_n^* \rangle +
\|x_n-x\|^2+ \langle x -y_n,J(x_n-x)\rangle + \langle y_n, y_n^*\rangle.
\end{align*}
By Fact \ref{caracliminf}\,(1), the sequences $(x_n)$ and $(J(x_n-x))$
are bounded, so passing to the limit in the above inequality we get
$$
\limsup \varphi_{T_n}(y_n,y_n^*)  \ge
\limsup \|x_n-x\|^2+ \langle x, x^*\rangle.
$$
Combining with \eqref{one}, we obtain
$$
\langle x, x^*\rangle \ge
\limsup \|x_n-x\|^2+ \langle x, x^*\rangle.
$$
This shows that $x_n \to x$.
The proof of (1) is complete.
\smallbreak
(2) 
We know from (1) that $\liminf T_n=L(f)$.
Since  $T_n=L(f_n)$ is maximal monotone, the sequence $(f_n)$
lies in  $\mathcal{F^*}(X\times X^*)$
by Fact \ref{BS03}. Hence $f\in \mathcal{F^*}(X\times X^*)$
by Proposition \ref{convergence}.
We conclude from Fact \ref{BS03} again
that $\liminf T_n=L(f)$ is maximal monotone.
\end{proof}

It may happen that the lower limit $\liminf T_n$ is not maximal monotone,
in which case according to Theorem \ref{reussi12017}\,(2) there is no
Mosco-converging representative functions of $T_n$. Here is an example:

\smallbreak
\noindent{\it Example} (taken from \cite{GL12}).
In $\R\times\R$, let $T_n=\{0\}\times \R$ for even $n$, $T_n=\R \times \{0\}$ for odd $n$.
Then, each $T_n$ is a maximal monotone operator with convex graph,
the lower limit $\liminf T_n = \{(0,0)\}$
is certainly representable but not maximal monotone; so, there is no
representative functions of $T_n$ which is epi-converging (=Mosco-converging).

\smallbreak
In fact, for a sequence of maximal monotone operators with convex graph
like in the above example,
the maximal monotonicity of the lower limit can happen only if
the sequence of operators can be represented by an epi-converging
sequence of representative functions. More precisely:

\begin{theo}\label{CN-maxmono}
Let $X$ be a Banach space space and let
$(T_n)\subset X\times X^*$ be a sequence of maximal monotone operators with convex graph.
If $T:=\liminf T_n$ is maximal monotone, then
the sequence $(\varphi_{T_n} + \delta_{T_n})$ of representative functions of $(T_n)$
epi-converges to $\varphi_{T} + \delta_{T}$, a representative function of $T$.
\end{theo}
\begin{proof}
Since $T_n$ is a maximal monotone operator with convex graph, the function
$f_n := \varphi_{T_n} + \delta_{T_n}$ is indeed a representative function of $T_n$.
We show that the sequence $(f_n)$ epi-converges to $f = \varphi_{T} + \delta_{T}$,
that is, for every $(x,x^*) \in X\times X^*$,

(i) there exists $(x_n,x_n^*)\to (x,x^*)$ with
$f(x,x^*)=\displaystyle\lim_n f_n(x_n,x_n^*)$;

(ii) for every $(x_n,x_n^*)\to (x,x^*)$, we have
$f(x,x^*)\le \displaystyle\liminf_n f_n(x_n,x_n^*)$.\\
We consider two cases.
\smallbreak
\textit{Case 1:} $(x,x^*) \notin T$. Let $((x_n,x_n^*))\subset X\times X^*$
be any sequence converging to $(x,x^*)$.
Since $T$ is maximal monotone,
there exists $(y,y^*) \in T$ such that $\langle y-x,y^*-x^*\rangle<0$.
But $T=\liminf_n T_n$, so for each $n\in \N$ there exists $(y_n,y_n^*) \in T_n$
such that $(y_n,y_n^*)\to (y,y^*)$. By continuity of the duality product,
we derive that $\langle y_n-x_n,y^*_n-x_n^*\rangle<0$ for large $n$, which implies that
$(x_n,x_n^*) \notin T_n$ for large $n$, because $T_n$ is monotone.
Therefore,
\[
\lim_n f_n(x_n,x_n^*) = +\infty = f(x, x^*).
\]
This shows that (i) and (ii) are satisfied.

\smallbreak
\textit{Case 2:} $(x,x^*) \in T$. Assertion (i) is satisfied: indeed,
for each $n\in \N$, there exists  $(x_n,x_n^*) \in T_n$ such that
$(x_n,x_n^*) \to  (x,x^*)$, hence
\[
f_n(x_n,x_n^*) = \langle x_n^*, x_n \rangle  \to \langle x^*, x \rangle = f(x,x^*).
\]
To prove Assertion (ii), let $(x_n,x_n^*)\to (x,x^*)$.
By definition of $f_n$, we have
\begin{equation}\label{one2}
f_n(x_n,x_n^*) \geq \varphi_{T_n}(x_n,x_n^*) =
\sup_{(y,y^*) \in T_n} \langle y-x_n,x^*_n-y^*\rangle +\langle x_n, x^*_n\rangle.
\end{equation}
Fix $(y,y^*) \in T$. Then, for each $n\in \N$ there exists $(y_n,y_n^*) \in T_n$
such that $(y_n,y_n^*) \to (y,y^*)$. It follows from (\ref{one2}) that
\[
f_n(x_n,x_n^*) \geq \langle y_n-x_n,x^*_n-y_n^*\rangle +\langle x_n, x^*_n\rangle,
\]
hence
\[
\liminf_{n\to \infty} f_n(x_n,x_n^*) \geq
\langle y-x,x^*-y^*\rangle +\langle x, x^*\rangle.
\]
Since $(y,y^*) \in T$ was taken arbitrarily, we conclude that
\[
\liminf_{n\to \infty} f_n(x_n,x_n^*) \geq \varphi_{T}(x, x^*).
\]
But $\varphi_{T}(x, x^*)=f(x, x^*)$, because $(x,x^*) \in T$, so Assertion (ii) holds.
\end{proof}

Combining the two previous theorems, we immediately obtain:

\begin{theo}
Let $X$ be a finite-dimensional space and let
$(T_n)\subset X\times X^*$ be a sequence of maximal monotone operators with convex graph.
Then, $\liminf T_n$ is maximal monotone if and only if
there exists a sequence $(f_n)$ of representative functions of $(T_n)$
which epi-converges.
\end{theo}

\section{Representation through subdifferentials}
The \textit{subdifferential} of $f\in \Gamma(X)$ is the set-valued operator
$\partial f:X\tos X^*$ defined by
$$
\partial f(x):=\{ x^*\in X^* : \langle x, x^*\rangle=f(x)+f^*(x^*) \}.
$$
Given $h \in\Gamma(X \times X^*)$, Fitzpatrick \cite{Fit88}
considered the operator $T_h\subset X \times X^*$ defined by
$$T_h:= \{(x,x^*)\in X \times X^* : (x^*, x)\in\partial h(x,x^*)\}.$$
It is readily seen that the function
$g(x,x^*):=(1/2)(h(x,x^*)+h^*(x^*,x))$ belongs to $\F(X\times X^*)$
and that $T_h=L(g)$.
Fitzpatrick \cite[Corollary 3.5]{Fit88} notes that
every maximal monotone operator $T$ can be represented this way,
namely $T=T_{\varphi_T}$.

\begin{theo}\label{limit2}
Let $X$ be a reflexive Banach space and let $(f_n) \subset \Gamma(X\times X^*)$
with $T_n:=T_{f_n}$ maximal monotone.
If $(f_n)$ Mosco-converges to $f$, then $T_f=\liminf T_n$ and $\liminf T_n$
is maximal monotone.
\end{theo}

\begin{proof} 
Let $g_n(x,x^*):=(1/2)(f_n(x,x^*)+f^*_n(x^*,x))$. Then,
$g_n\in \F(X\times X^*)$, $T_n=L(g_n)$ is maximal monotone
and by Fact \ref{Mos71}, the sequence $(g_n)$ Mosco converges to the function
$g(x,x^*):=(1/2)(f(x,x^*)+f^*(x^*,x))$.
Hence, by Theorem \ref{reussi12017},
$\liminf T_n=L(g)=T_f$, and $\liminf T_n$ is
maximal monotone.
\end{proof}

\begin{cor}[Attouch \cite{Att77}] Let $X$ be a reflexive Banach space and
let $(g_n) \subset \Gamma(X)$.
If $(g_n)$ Mosco-converges to $g$,
then
$\partial g =\liminf \partial g_n$.
\end{cor}
\begin{proof}
Let $f_n(x,x^*)= g_n(x)+g_n^*(x^*)$. Then $f_n\in \Gamma(X\times X^*)$ and
$T_{f_n}=\partial g_n$ (\cite[Example 3]{Fit88}).
By Fact \ref{Mos71}, $(f_n)$ Mosco-converges to $f$ given by
$f(x,x^*)= g(x)+g^*(x^*)$. Since $T_{f_n}=\partial g_n$ is maximal
monotone, we may apply Theorem \ref{limit2} to derive
that $T_f=\liminf T_n$. As above $T_f=\partial g$.
In other words, $\partial g =\liminf \partial g_n$.
\end{proof}

{\small

\end{document}